\documentclass[12pt, reqno]{amsart}
\usepackage{amsmath, amsthm, fullpage}

\newtheorem{thm}{Theorem}
\newtheorem{lemma}[thm]{Lemma}
\newtheorem{theorem}[thm]{Theorem}
\newtheorem{corollary}[thm]{Corollary}

\newcommand{\al}{\alpha}
\renewcommand{\phi}{\varphi}

\newcommand{\R}{{\mathbb R}}

\renewcommand{\qed}{\rule{3mm}{3mm}}
\renewenvironment{proof}
    {\vspace{1mm}\noindent\textbf{Proof.}}
    {\hspace*{\fill} $\qed$\vspace{1mm}}
\newenvironment{proof_of}[1]
    {\vspace{1mm}\noindent {\bf Proof of #1.}}
    {\hspace*{\fill} $\qed$\vspace{1mm}}

\begin{document}
\title{A lower bound for the amplitude of\\ traveling waves of suspension bridges}
\author{Paschalis Karageorgis \and John Stalker}
\address{School of Mathematics, Trinity College, Dublin 2, Ireland}
\email{pete@maths.tcd.ie} \email{stalker@maths.tcd.ie}

\begin{abstract}
We obtain a lower bound for the amplitude of nonzero homoclinic traveling wave solutions of the McKenna--Walter suspension bridge
model. As a consequence of our lower bound, all nonzero homoclinic traveling waves become unbounded as their speed of propagation
goes to zero (in accordance with numerical observations).
\end{abstract}
\maketitle

We study traveling wave solutions of the McKenna--Walter suspension bridge model
\begin{equation}\label{pde}
u_{tt} + u_{xxxx} + f(u) = 0
\end{equation}
introduced in \cite{mcwa}.  Two of the most standard choices for the nonlinear term $f$ are
\begin{equation}\label{nonl}
f(u) = \max(u,-1), \qquad f(u) = e^u - 1.
\end{equation}
The piecewise linear choice was the original choice in \cite{mcwa} and stems from the fact that the cables of a suspension bridge
will resist movement only in one direction.  The exponential choice was introduced in \cite{chemck} as a smooth version of the
piecewise linear one which behaves the same way as $u\to -\infty$ and also near $u=0$.  The smooth version is more suitable when
it comes to numerics and also seems more appealing to engineers \cite{ds}.

Numerical results for \eqref{pde} go back to McKenna and Walter \cite{mcwa} who studied traveling wave solutions.  Those are
solutions of the form $u=u(x-ct)$, so they satisfy the ODE
\begin{equation}\label{ode}
u'''' + c^2 u'' + f(u) = 0.
\end{equation}
Based on the numerical results of \cite{mcwa} for the piecewise linear model and those of \cite{chemck, chmc} for the exponential
one, traveling waves become unbounded as $c\to 0$.  A rigorous proof of this observation was given by Lazer and McKenna
\cite{lamc}, but it only applied to the piecewise linear model and not the exponential one.

In this paper, we give a simple proof that applies to any nonlinear term $f$ such that
\begin{itemize}
\item[(A1)] $f$ is locally Lipschitz continuous with $uf(u)>0$ for all $u\neq 0$, and
\item[(A2)] $f$ is differentiable at the origin with $f'(0)>0$.
\end{itemize}
Clearly, these assumptions hold for both nonlinearities in \eqref{nonl}.  To show that traveling waves become unbounded as their
speed goes to zero, we actually prove a lower bound for their amplitude, which is of independent interest by itself.  In what
follows, we shall only focus on homoclinic solutions, namely those which vanish at $\pm\infty$.

\begin{theorem}\label{main}
Assume (A1)-(A2) and that $0<c^4<4f'(0)$.  If $u$ is a nonzero homoclinic solution of equation
\eqref{ode}, then $||u||_\infty\geq L(f,c)$, where
\begin{equation}\label{lb}
L(f,c) = \sup \left\{ \delta >0: \frac{f(u)}{u} > \frac{c^4}{4} \:\,\text{whenever\, $0\neq |u| < \delta$} \right\}.
\end{equation}
\end{theorem}

We remark that the lower bound $L(f,c)$ is well-defined because
\begin{equation*}
\lim_{u\to 0} \frac{f(u)}{u} = f'(0) > \frac{c^4}{4}.
\end{equation*}
Moreover, $L(f,c)\to\infty$ as $c\to 0$ by (A1), so nonzero homoclinic solutions of \eqref{ode} become unbounded as $c\to 0$. Our
assumption that $c^4<4f'(0)$ is natural because the eigenvalues of the linearized problem become purely imaginary when $c^4\geq
4f'(0)$, so one does not expect homoclinic solutions in that case.

The existence of homoclinic solutions of \eqref{ode} has been studied by several authors.  We refer the reader to \cite{chemck}
for the piecewise linear case, \cite{kamc} for more general nonlinearities that grow polynomially and \cite{sw} for the
exponential case.  The authors of \cite{BFGK} studied the qualitative properties of solutions for nonlinearities that satisfy
(A1).  There is also an existence result by Levandosky \cite{lev1} when $f(u)= u-|u|^{p-1}u$, but this case is somewhat different
because it does not satisfy (A1) and the solutions remain bounded as $c\to 0$.

To prove Theorem \ref{main}, we shall need to use the following facts from \cite{BFGK}.  We only give the proof of the last two
parts and refer the reader to \cite[Proposition 11]{BFGK} for the first.

\begin{lemma}
Suppose $u$ is a nonzero homoclinic solution of equation \eqref{ode}, namely a solution of equation \eqref{ode} that vanishes at
$\pm\infty$.
\begin{itemize}
\item[(a)] Assume $f$ is continuous with $f(0)=0$.  Then $u',u'',u'''$ must also vanish at $\pm\infty$.
\item[(b)] Assume (A1).  Then $u(s)$ must change sign infinitely many times as $s\to\pm \infty$.
\item[(c)] Assume (A1)-(A2) and that $0<c^4<4f'(0)$.  Then $u\in H^2$.
\end{itemize}
\end{lemma}

\begin{proof}
To prove (b), we note that $u(s)$ is a solution of \eqref{ode} if and only if $u(-s)$ is.  Thus, it suffices to show that $u(s)$
changes sign infinitely many times as $s\to \infty$.  Suppose $u(s)$ is eventually non-negative, the other case being similar.
Then $w(s)= u''(s) + c^2 u(s)$ satisfies
\begin{align*}
w''(s) = -f(u(s))\leq 0
\end{align*}
for large enough $s$, so $w(s)$ is eventually concave.  Since $w(s)$ goes to zero by part (a), this implies that $w(s)$ is
eventually non-positive.  Using this fact, we now get
\begin{equation*}
u''(s) = w(s) - c^2 u(s) \leq 0
\end{equation*}
for large enough $s$, so $u(s)$ is eventually concave.  As before, this implies $u(s)$ is eventually non-positive, so we must
actually have $u\equiv 0$, a contradiction.

To prove (c), we consider the function
\begin{equation}\label{H}
H(s) = u'(s)u''(s) -u(s)u'''(s) - c^2 u(s) u'(s).
\end{equation}
We note that $H(s)$ is bounded by part (a) and that a short computation gives
\begin{equation}\label{H2}
H(s_2) - H(s_1) = \int_{s_1}^{s_2} [u''(s)^2 - c^2 u'(s)^2 + u(s) f(u(s))] \,ds
\end{equation}
for all $s_1<s_2$.  Now, fix some $0<\varepsilon< f'(0) -\frac{c^4}{4}$ and let $s_0\in\R$ be such that
\begin{equation*}
|s|\geq s_0 \quad\Longrightarrow\quad \frac{f(u(s))}{u(s)} \geq f'(0) -\varepsilon.
\end{equation*}
Recalling part (b), suppose $s_1,s_2$ are any two roots of $u(s)$ for which $|s_1|,|s_2|> |s_0|$.  Then we may combine the last
two equations to find that
\begin{align} \label{Hin}
H(s_2) - H(s_1) &\geq \int_{s_1}^{s_2} \bigl[ u''(s)^2 - c^2 u'(s)^2 + (f'(0) -\varepsilon) u(s)^2 \bigr] \,ds \\
&\geq -2\sqrt{f'(0)-\varepsilon} \int_{s_1}^{s_2} u''(s) u(s) \,ds - c^2 \int_{s_1}^{s_2} u'(s)^2 \:ds \notag \\
&= \al \int_{s_1}^{s_2} u'(s)^2 \,ds \notag,
\end{align}
where $\al = 2\sqrt{f'(0)-\varepsilon} - c^2>0$.  Since $H(s)$ is bounded by above, this implies $u'\in L^2$.  Using this fact
and the inequality \eqref{Hin}, we conclude that $u\in H^2$.
\end{proof}

\begin{proof_of}{Theorem \ref{main}}
We multiply equation \eqref{ode} by $u$ and integrate by parts to get
\begin{equation*}
\int_{-\infty}^\infty u''(s)^2 \,ds - c^2 \int_{-\infty}^\infty u'(s)^2 \,ds + \int_{-\infty}^\infty uf(u) \,ds = 0.
\end{equation*}
Using the Fourier transform and a trivial estimate, we conclude that
\begin{equation}\label{useful}
\int_{-\infty}^\infty uf(u) \,ds = c^2 \int_{-\infty}^\infty u'(s)^2 \,ds - \int_{-\infty}^\infty u''(s)^2 \,ds
\leq \frac{c^4}{4} \int_{-\infty}^\infty u(s)^2 \,ds.
\end{equation}

Since $\lim_{u\to 0} \frac{f(u)}{u} = f'(0) > \frac{c^4}{4}$, we can always find some $\delta>0$ such that
\begin{equation*}
0\neq |u| < \delta \quad\Longrightarrow\quad u f(u) > \frac{c^4 u^2}{4}.
\end{equation*}
Suppose $\delta>0$ is any number with this property.  Given a nonzero homoclinic solution $u$ for which $||u||_\infty<
\delta$, we can then use the last equation to get
\begin{equation*}
\int_{-\infty}^\infty uf(u) \,ds > \frac{c^4}{4} \int_{-\infty}^\infty u(s)^2 \,ds,
\end{equation*}
contrary to \eqref{useful}.  This means that $||u||_\infty\geq \delta$ for each nonzero homoclinic solution and each such
$\delta$, so the result follows.
\end{proof_of}

The following corollary is a trivial consequence of our estimate \eqref{useful}.  This result refines \cite[Theorem 13i]{BFGK}
which imposes the stronger assumption $\frac{f(u)}{u} \geq f'(0)$.

\begin{corollary}
Assume (A1)-(A2) and that $0<c^4<4f'(0)$.  If $\frac{f(u)}{u} > \frac{c^4}{4}$ for all $u\neq 0$, then equation \eqref{ode} has
no nonzero homoclinic solutions.
\end{corollary}

\bibliographystyle{siam}
\bibliography{biblio}
\end{document}